\documentclass{amsart}
\usepackage{amssymb}
\theoremstyle{plain}
\newtheorem{thm}{Theorem}[section]

\newtheorem{cor}[thm]{Corollary}
\theoremstyle{definition}

\newcommand{\Z}{\mathbb{Z}}
\newcommand{\N}{\mathbb{N}}

\DeclareMathOperator{\GL}{GL}

\DeclareMathOperator{\Ker}{Ker}
\DeclareMathOperator{\Ima}{Im}
\DeclareMathOperator{\Def}{Def}
\DeclareMathOperator{\Ind}{Ind}

\begin{document}
\title[Kac's Theorem for equipped graphs]{Kac's Theorem for equipped graphs and for maximal rank representations}

\author{William Crawley-Boevey}
\address{Department of Pure Mathematics, University of Leeds, Leeds LS2 9JT, UK}
\email{w.crawley-boevey@leeds.ac.uk}

% 16G20 Representations of quivers and partially ordered sets

\begin{abstract}
We give two generalizations of Kac's Theorem on representations of quivers.
One is to representations of equipped graphs by relations,
in the sense of Gelfand and Ponomarev.
The other is to representations of quivers in which certain
of the linear maps are required to have maximal rank.
\end{abstract}
\maketitle

\section{Introduction}
Recall that Gabriel's Theorem \cite{Gab} asserts that a quiver has only
finitely many indecomposable representations up to isomorphism if and only if the underlying graph is a Dynkin diagram,
and that in this case the indecomposables are in \mbox{1-1} correspondence with positive roots for the corresponding
root system. Kac's Theorem \cite{Kac1,Kac2} extends the latter part to arbitrary quivers, saying that the
possible dimension vectors of indecomposable representations are the positive roots, with a unique
indecomposable for real roots and infinitely many for imaginary roots.

In \cite{GP}, Gelfand and Ponomarev announced a generalization of Gabriel's Theorem to representations of
what they called \emph{equipped graphs},
which may be thought of as generalizations of quivers, in which one is allowed not just arrows $\bullet\longrightarrow\bullet$,
but also edges with two heads $\bullet\leftarrow\!\rightarrow\bullet$ or two tails
$\bullet-\!\!\!-\!\!\!-\bullet$.
%$\bullet\rightarrow\!\leftarrow\bullet$.
Representations are given by a vector space at each vertex and a linear relation of a certain sort for each edge.
One motivation, perhaps, was that by taking all edges to be two-headed or two-tailed, one can study the
representation theory of a graph without needing to choose an orientation.

In this paper we generalize Kac's Theorem to equipped graphs. In order to obtain our proof, we study representations of quivers
in which the linear maps associated to certain arrows are required to have maximal rank, and also prove an
analogue of Kac's Theorem in this setting. The case of real roots has been dealt with by Wiedemann \cite{Wie}.

\section{Maximal rank representations}

Let $Q = (Q_0,Q_1,h,t)$ be a quiver and $\alpha\in\N^{Q_0}$ a dimension vector.
Recall that if $q$ is a prime power, $A_{Q,\alpha}(q)$ denotes the number of
isomorphism classes of absolutely indecomposable representations of $Q$
of dimension vector $\alpha$ over the finite field with $q$ elements.
By Kac's Theorem \cite[\S 1.15]{Kac2}, $A_{Q,\alpha}(q)$ is given by a polynomial in $\Z[q]$,
which is non-zero if and only if $\alpha$ is a positive root, and then is monic of
degree $1-q_Q(\alpha)$, where $q_Q$ is the associated quadratic form.

If $M$ is a subset of $Q_1$, we say that a representation of $Q$ is \emph{$M$-maximal} if
the arrows in $M$ are represented by linear maps of maximal rank. We denote by $A^M_{Q,\alpha}(q)$ the number of isomorphism
classes of absolutely indecomposable $M$-maximal representations of $Q$ over the field with $q$ elements.

\begin{thm}
$A^M_{Q,\alpha}(q)$ is given by a polynomial in $\Z[q]$,
which is non-zero if and only if $\alpha$ is a positive root, and then is monic of degree $1-q_Q(\alpha)$.
Moreover this polynomial does not depend on the orientation of $Q$.
\end{thm}

\begin{proof}
In case $M$ is empty, this is Kac's Theorem, so suppose $M$ is not empty.
We prove the assertion for all quivers, dimension vectors and subsets $M$ by induction,
firstly on $m = \max\{\min(\alpha_{t(a)},\alpha_{h(a)}):a\in M\}$,
and then on the size of the set $\{ a\in M:\min(\alpha_{t(a)},\alpha_{h(a)})=m\}$.

Choose $a\in M$ with $\min(\alpha_{t(a)},\alpha_{h(a)})=m$ and let $M'=M\setminus\{a\}$.
Let $Q'$ be the quiver obtained from $Q$ by adjoining a new vertex $v$ and replacing $a$
by arrows $b:t(a)\to v$ and $c:v\to h(a)$.
By factorizing a linear map as the composition of a surjection followed by an injection, one sees
that there is a 1-1 correspondence, up to isomorphism, between representations of $Q$ in which $a$ has rank $d$,
and representations of $Q'$ in which the vector space at $v$ has dimension $d$ and $b$ and $c$ have rank $d$.
Moreover this correspondence respects absolute indecomposability.
For any $0\le d < m$, let $\alpha_d$ be the dimension vector for $Q'$ obtained from $\alpha$ by setting $\alpha_d(v) = d$.
Using that in any representation, if $a$ does not have rank $m$, then it must have rank $d<m$, we obtain
\[
A^M_{Q,\alpha}(q) = A^{M'}_{Q,\alpha}(q) - \sum_{d=0}^{m-1} A^{M'\cup\{b,c\}}_{Q',\alpha_d}(q).
\]
By induction all terms on the right hand side are independent of orientation and in $\Z[q]$, and (as used in \cite{Wie}),
\begin{align*}
q_{Q'}(\alpha_d)
&= q_Q (\alpha) + \alpha_{h(a)}\alpha_{t(a)} + d^2 - d\alpha_{t(a)} - \alpha_{h(a)}d \\
&= q_Q(\alpha) + (\alpha_{h(a)} - d)(\alpha_{t(a)} - d) > q_Q(\alpha)
\end{align*}
so that only $A^{M'}_{Q,\alpha}(q)$, which is monic of degree $1-q_Q(\alpha)$, contributes to the leading term of
$A^M_{Q,\alpha}(q)$.
\end{proof}

Let $K$ be an algebraically closed field. The theorem, together with standard arguments as used in the proof of Kac's Theorem,
give the following result.

\begin{cor}
There is an indecomposable $M$-maximal representation of $Q$ of dimension vector $\alpha$ over $K$
if and only if $\alpha$ is a positive root.
If $\alpha$ is a real root, there is a unique such representation up to isomorphism,
while if $\alpha$ is an imaginary root, there are infinitely many such representations.
\end{cor}

\section{Equipped graphs}

As we will allow our equipped graphs to have loops, our definitions differ slightly from those of Gelfand and Ponomarev \cite{GP}.

Recall that a graph may be thought of as a pair $(G,*)$ consisting of a quiver $G=(G_0,G_1,h,t)$, together with
a fixed-point free involution $*$ of the set of arrows $G_1$ which exchanges heads and tails,
so $h(a)=t(a^*)$ for all $a\in G_1$.
By an \emph{equipped graph} we mean a collection $((G,*),\phi)$ consisting of a graph $(G,*)$ and a mapping $\phi:G_1\to\{0,1\}$.

One may depict an equipped graph by drawing a vertex for each element of $G_0$, an arrow $a:t(a)\longrightarrow h(a)$
for each $a\in G_1$ with $\phi(a)=1$ and $\phi(a^*)=0$, a double-headed edge $a:t(a)\leftarrow\!\rightarrow h(a)$
for each $*$-orbit $\{a,a^*\}$ with
$\phi(a)=\phi(a^*)=0$, and a double-tailed edge
%$a:t(a)\rightarrow\!\leftarrow h(a)$
$a:t(a) -\!\!\!-\!\!\!- h(a)$
for each $*$-orbit $\{a,a^*\}$ with $\phi(a)=\phi(a^*)=1$.

Recall that a \emph{linear relation} between vectors spaces $V$ and $W$ is a subspace $R \subseteq V\oplus W$.
The \emph{opposite relation} is $R^{op} = \{(w,v):(v,w)\in R\} \subseteq W\oplus V$.
Considering graphs of linear maps leads one to define the following subspaces:
the kernel $\Ker R = \{v\in V:(v,0)\in R\}$,
the domain of definition $\Def R = \{v\in V:\text{$(v,w)\in R$ for some $w\in W$}\}$,
the indeterminacy $\Ind R = \{w\in W:(0,w)\in R\}$,
and the image $\Ima R = \{w\in W:\text{$(v,w)\in R$ for some $v\in V$}\}$.
There is an induced isomorphism $\Def R/\Ker R\cong \Ima R/\Ind R$, and we call the dimension of this space the \emph{rank} of $R$.
%Observe that to fix a linear relation $R \subseteq V\oplus W$ is the same thing as fixing subspaces $\Ker R \subseteq \Def R\subseteq V$
%and $\Ind R \subseteq \Ima R\subset W$ and an isomorphism $\Def R/\Ker R\to \Ima R/\Ind R$.

Following Gelfand and Ponomarev, if $r,s\in\{0,1\}$, we say that $R$ is an \emph{equipped relation of type $(r,s)$} if
in $V$, either $\Ker R=0$ if $r=0$ or $\Def R = V$ if $r=1$,
and in $W$, either $\Ind R=0$ if $s=0$ or $\Ima R=W$ if $s=1$.
As observed by Gelfand and Ponomarev, an equipped relation of type $(1,0)$ is the same as the graph of a linear map,
and one of type $(0,1)$ is the opposite of a graph of a linear map. On the other hand,
specifying an equipped relation
of type $(0,0)$ is the same as specifying a pair of subspaces $\Def R \subseteq V$ and $\Ima R\subseteq W$ and an isomorphism $\Def R\cong \Ima R$,
and specifying an equipped relation
of type $(1,1)$ is the same as specifying a pair of subspaces $\Ker R \subseteq V$ and $\Ind R\subseteq W$ and an isomorphism $V/\Ker R\cong W/\Ind R$.

By a \emph{representation} $X$ of an equipped graph $((G,*),\phi)$ we mean a collection consisting of a vector space $X_i$
for each vertex $i\in G_0$ and an equipped relation $X_a \subset X_{t(a)}\oplus X_{h(a)}$ of type $(\phi(a),\phi(a^*))$
for each arrow $a\in G_1$, such that $X_{a^*} = X_a^{op}$ for all $a$.

A \emph{morphism} of representations $\theta:X\to X'$ is given by a collection of linear maps $\theta_i:X_i\to X'_i$ for each $i\in G_0$,
such that $(\theta_{t(a)}(x),\theta_{h(a)}(y))\in X'_a$ for all $a\in G_1$ and $(x,y)\in X_a$.
%\[
%\begin{pmatrix}
%\theta_{t(a)} & 0 \\
%0 & \theta_{h(a)}
%\end{pmatrix} (X_a) \subseteq X'_a
%\]
%for all $a\in G_1$.
In this way one obtains an additive category of representations of the equipped graph $((G,*),\phi)$ over any given field.
It is easy to see that this category has split idempotents, and
we can clearly also define the notion of an absolutely indecomposable representation.
The \emph{dimension vector} of a representation $X$ is the tuple $(\dim X_i) \in\N^{G_0}$.

\begin{thm}
The number of isomorphism classes of absolutely indecomposable representations of an
equipped graph $((G,*),\phi)$
of dimension vector $\alpha\in\N^{G_0}$ over a finite field with $q$ elements is equal to $A_{Q,\alpha}(q)$, where $Q$ is any quiver with underlying
graph $(G,*)$.
\end{thm}

\begin{proof}
We use the following observation. If $V$ and $W$ are vector spaces, and $U$ is a vector space of dimension $d$,
then equipped relations $R\subseteq V\oplus W$ of type $(r,s)$ and rank $d$
are in 1-1 correspondence with $\GL(U)$-orbits of pairs of linear maps $a,b$ of rank $d$ in the following configurations.
\[
\begin{tabular}{|c|c|l|}
\hline
Type & Configuration & Corresponding equipped relation
\\
\hline
$(0,0)$ & $V\stackrel{a}\hookleftarrow U\stackrel{b}\hookrightarrow W$ & $R = \{(a(u),b(u)):u\in U\}$
\\
$(0,1)$ & $V\stackrel{a}\hookleftarrow U\stackrel{b}\twoheadleftarrow W$ & $R = \{(a(b(w)),w):w\in W\}$
\\
$(1,0)$ & $V\stackrel{a}\twoheadrightarrow U\stackrel{b}\hookrightarrow W$ & $R = \{(v,b(a(v))):v\in V\}$
\\
$(1,1)$ & $V\stackrel{a}\twoheadrightarrow U\stackrel{b}\twoheadleftarrow W$ & $R = \{(v,w)\in V\oplus W: a(v)=b(w)\}$.
\\
\hline
\end{tabular}
\]
%
%\[
%\begin{matrix}
%\text{Type} & \text{Configuration} & \text{Corresponding equipped relation}\\
%(0,0) & V\xleftarrow{a} U\xrightarrow{b} W & R = \{(a(u),b(u)):u\in U\}\\
%(0,1) & V\xleftarrow{a} U\xleftarrow{b} W & R = \{(a(b(w)),w):w\in W\}\\
%(1,0) & V\xrightarrow{a} U\xrightarrow{b} W & R = \{(v,b(a(v))):v\in V\} \\
%(1,1) & V\xrightarrow{a} U\xleftarrow{b} W & R = \{(v,w)\in V\oplus W: a(v)=b(w)\}
%\end{matrix}
%\]
%\[
%\begin{matrix}
%\text{Type} & \text{Configuration} & \text{Corresponding equipped relation}\\
%(0,0) & V\stackrel{a}\hookleftarrow U\stackrel{b}\hookrightarrow W & R = \{(a(u),b(u)):u\in U\}\\
%(0,1) & V\stackrel{a}\hookleftarrow U\stackrel{b}\twoheadleftarrow W & R = \{(a(b(w)),w):w\in W\}\\
%(1,0) & V\stackrel{a}\twoheadrightarrow U\stackrel{b}\hookrightarrow W & R = \{(u,b(a(v))):v\in V\} \\
%(1,1) & V\stackrel{a}\twoheadrightarrow U\stackrel{b}\twoheadleftarrow W & R = \{(v,w)\in V\oplus W: a(v)=b(w)\}.
%\end{matrix}
%\]

Let $\Delta$ be the quiver whose
vertex set is the disjoint union $\Delta_0 = G_0 \cup \overline G_1$,
where $\overline G_1$ denotes the set of $*$-orbits $[a]$ in $G_1$,
and where each arrow $a\in G_1$ defines an arrow $a'$ in $\Delta_1$,
with either $h(a') = t(a), t(a') = [a]$ if $\phi(a)=0$, or $h(a') = [a], t(a') = t(a)$ if $\phi(a)=1$.

The observation above leads to an
equivalence between the category of representations of $((G,*),\phi)$
and the category of representations of $\Delta$ in which arrows
with head in $\overline G_1$ are represented by surjective linear maps
and arrows with tail in $\overline G_1$ are represented by injective linear maps.

It follows that the number of isomorphism classes of absolutely indecomposable representations of
$((G,*),\phi)$ of dimension vector $\alpha\in\N^{G_0}$ over a finite field with $q$ elements is
\[
n_{((G,*),\phi),\alpha}(q) = \sum_{\alpha'} A^{\Delta_1}_{\Delta,\alpha'}(q),
\]
where $\alpha'$ runs over all dimension vectors for $\Delta$ whose restriction to $G_0$ is equal to $\alpha$
and with $\alpha'_{[a]} \le \min(\alpha_{t(a)},\alpha_{h(a)})$ for all $a\in G_1$.

Changing the equipping function $\phi$, changes the orientation of $\Delta$,
but by the results in the previous section, this does not change the sum above,
so $n_{((G,*),\phi),\alpha}(q)$ is independent of $\phi$.

In particular, identifying $G_1$ with the disjoint union $Q_1 \cup Q_1^*$,
and defining $\phi'$ by $\phi'(a) = 1\Leftrightarrow a\in Q_1$,
representations of $(G,*,\phi')$ are essentially the same as representations of $Q$, so
$n_{((G,*),\phi),\alpha}(q) = n_{((G,*),\phi'),\alpha}(q) = A_{Q,\alpha}(q)$.
\end{proof}

If $K$ is an algebraically closed field, we deduce the following.

\begin{cor}
There is an indecomposable representation of an equipped graph $((G,*),\phi)$ of dimension
vector $\alpha\in \N^{G_0}$ over $K$ if and only if $\alpha$ is a positive root for the underlying graph $(G,*)$.
If $\alpha$ is a real root, there is a unique such representation up to isomorphism,
while if $\alpha$ is an imaginary root, there are infinitely many such representations.
\end{cor}

\end{document}